\documentclass[9pt]{article}%
\usepackage{bbm}
\usepackage{epsfig}
\usepackage{graphics,graphicx,amssymb,amsmath,verbatim}
\usepackage{amssymb}
\usepackage{mathrsfs}
\usepackage{amsfonts}
\usepackage{amsmath}
\usepackage{graphicx}
\usepackage{easybmat}
\usepackage{color}%
\providecommand{\U}[1]{\protect \rule{.1in}{.1in}}
\newtheorem{theorem}{Theorem}

\newtheorem{corollary}{Corollary}

\newtheorem{definition}{Definition}

\newtheorem{lemma}{Lemma}

\newtheorem{remark}{Remark}

\newenvironment{proof}[1][Proof]{\noindent \textbf{#1.} }{\  \rule{0.5em}{0.5em}}
\definecolor{blue}{rgb}{0,0,1}
\allowdisplaybreaks[4]
\begin{document}

\title{Stability Analysis of Nonlinear Time-Varying Systems by Lyapunov Functions
with Indefinite Derivatives}
\author{Bin Zhou\thanks{Center for Control Theory and Guidance Technology, Harbin
Institute of Technology, Harbin, 150001, China. Email:
\texttt{binzhoulee@163.com, binzhou@hit.edu.cn}.}}
\date{}
\maketitle

\begin{abstract}
This paper is concerned with stability analysis of nonlinear time-varying
systems by using Lyapunov function based approach. The classical Lyapunov
stability theorems are generalized in the sense that the time-derivative of
the Lyapunov functions are allowed to be indefinite. The stability analysis is
accomplished with the help of the scalar stable functions introduced in our
previous study. Both asymptotic stability and input-to-state stability are
considered. \ Particularly, for asymptotic stability, several concepts such as
uniform and non-uniform asymptotic stability, and uniform and non-uniform
exponential stability are studied. The effectiveness of the proposed theorems
is illustrated by several numerical examples.

\vspace{0.3cm}


\end{abstract}

\section{Introduction}

It has been recognized in the literature that time-varying systems, which are
also referred to as non-autononuous systems and non-stationary systems, are
more difficult to handle than time-invariant systems
\cite{khalil02book,vid02book,zlf15auto,zhou15auto}. Let's take linear systems
for example. For linear time-invariant (LTI) system, it is known that there
are only two kinds of stability concepts, namely, Lyapunov stability and
asymptotic stability, which are totally determined by the eigenvalue set of
the system matrix (see, for example, \cite{rugh96book}). However, for linear
time-varying (LTV) system, there are more stability concepts such as
non-uniformly asymptotic stability, uniformly asymptotic stability,
non-uniformly exponential stability and uniformly exponential stability. See
our recent paper \cite{zhou14auto} for the detailed distinction among these
concepts. Moreover, differently from LTI systems, the stability of LTV systems
cannot be linked to the eigenvalue set of the system matrices directly
\cite{rugh96book,zhou14auto,ZCD13IJC}. Therefore, compared with time-invariant
systems, study on the analysis and design of time-varying systems is very
challenging and has been greatly retarded and only relatively less papers were
available in the literature (see
\cite{dqyl13auto,jtd12auto,mam15auto,zlf15auto,zhou14autoA} and the references therein).

Lyapunov indirect approach, which is also know as the Lyapunov's second
approach, is a powerful tool for stability analysis and design of control
systems \cite{khalil02book}. By this method, if a positive definite function
of the state can be found such that its time-derivative along the trajectories
of the considered system is negative definite, it is claimed that the system
is stable. Moreover, by imposing different positive definiteness assumptions
and different negative definiteness assumptions on the Lyapunov function and
its time-derivative, respectively, different stability properties of the
considered system can be deduced. Generally, for time-invariant systems, the
negative definiteness of the time-derivative of the Lyapunov function can be
relaxed as negative semi-definiteness, for which the so-called Lasalle
invariant principle can be utilized. For time-varying systems, expect for some
special cases, the Lasalle invariant principle is either not valid or
difficult to use \cite{vid02book}. Hence some researchers attempt to use the
available Lyapunov function, whose time-derivative is not strictly negative
definite, to construct a new Lyapunov function whose time-derivative is
negative definite \cite{mazenc03auto,mm05auto}.

Recently, for LTV systems we provided in \cite{zhou14auto} a new Lyapunov
function based stability analysis approach, which allows the time-derivative
of the Lyapunov function to be indefinite. The stability of the considered
system can be guaranteed if some scalar function is stable, and, different
stability properties of the scalar function together with different
assumptions on the bound of the Lyapunov functions gives different stability
outcomes of the considered system. The most advantages of this approach is
that, as we have emphasized, the time-derivative of the Lyapunov function is
neither required to be negative definite nor required to be negative
semi-definite. The idea in this approach was latterly extended to time-delay
systems in \cite{ze15auto} where both time-varying Razumikhin function and
time-varying Lyapunov-Krasovskii functional based stability theorems were built.

In this note we continue to extend the approach in \cite{zhou14auto} and
\cite{ze15auto} to the stability analysis of general nonlinear time-varying
systems, namely, we build stability theorems by using Lyapunov functions whose
time-derivative can take both positive and negative values. As done in
\cite{zhou14auto} and \cite{ze15auto}, this is achieved by introducing a
scalar stable function on the time-derivative of the Lyapunov function. Both
asymptotic stability and input-to-state (ISS) stability are considered. For
asymptotic stability, we provide two theorems. The first theorem (Theorem
\ref{th1}) can be used to claim non-uniformly asymptotic stability, uniformly
asymptotic stability, non-uniformly exponential stability, and uniformly
exponential stability, by imposing different assumptions on the scalar
function and the bounds for the Lyapunov function. The second stability
theorem (Theorem \ref{th3}) further allows a drifting term, which is
non-negative, in the time-derivative of the Lyapunov function, and the
asymptotic stability can be concluded if an additional condition is satisfied.
For ISS stability, we provide two theorems (Theorems \ref{th4} and \ref{th2})
for testing respectively the ISS stability and the integral input-to-state
(iIIS) stability. The advantages of the proposed theorems over the existing
results are pointed out and their effectiveness are also illustrated by
several numerical examples with some of which borrowed from the literature.

The remainder of this note is organized as follows. The system description and
preliminaries are given in Section \ref{sec2}. Our main results are given in
Section \ref{sec3} which contains two subsections dealing with asymptotic
stability and ISS stability, respectively. Numerical examples are given in
Section \ref{sec4}. The paper is concluded in Section \ref{sec5} and, finally,
some proofs are collected in the Appendix.

\textbf{Notation}: Throughout this note, if not specified, we let
$J=[t^{\#},\infty)$ with $t^{\#}$ being some scalar and use $\mathbb{C}%
^{1}\left(  J,\Omega \right)  $ and $\mathbb{PC}\left(  J,\Omega \right)  $ to
denote respectively the space of $\Omega$-valued continuously differentiable
functions and piecewise continuous functions defined on $J.$ We denote%
\begin{align*}
\mathcal{L}_{p}^{m}\left(  J\right)   &  =\left \{  f\left(  t\right)
:J\rightarrow \mathbf{R}^{m}\; \left \vert \left(  \int_{J}\left \vert f\left(
t\right)  \right \vert ^{p}\mathrm{d}t\right)  ^{\frac{1}{p}}<\infty \right.
\right \}  ,\\
\mathcal{L}_{\infty}^{m}\left(  J\right)   &  =\left \{  f\left(  t\right)
:J\rightarrow \mathbf{R}^{m}\; \left \vert \  \sup_{t\in J}\{ \left \vert f\left(
t\right)  \right \vert \}<\infty \right.  \right \}  ,\\
\mathcal{Z}^{m}\left(  J\right)   &  =\left \{  f\left(  t\right)
:J\rightarrow \mathbf{R}^{m}\; \left \vert \lim_{t\rightarrow \infty}\left \vert
f\left(  t\right)  \right \vert =0\right.  \right \}  ,
\end{align*}
where $p\in \lbrack1,\infty)$ is any integer and $\left \vert \cdot \right \vert $
refers to the usual Euclidean norm. Moreover, if $m=1,$ then the
$\mathcal{L}_{p}^{m}\left(  J\right)  ,\mathcal{L}_{\infty}^{m}\left(
J\right)  $ and $\mathcal{Z}^{m}\left(  J\right)  $ will be respectively
denoted by $\mathcal{L}_{p}\left(  J\right)  ,\mathcal{L}_{\infty}\left(
J\right)  $ and $\mathcal{Z}\left(  J\right)  $ for short. We also use
$\left \Vert f\right \Vert _{\left[  t_{0},t\right]  }=\sup \{ \left \vert
f\left(  t\right)  \right \vert ,t\in \left[  t_{0},t\right]  \subset J\}$ to
denote the truncation of the norm of $f$ at $t$. \ The following are
definitions of comparison functions \cite{khalil02book}.

\begin{itemize}
\item A function $\alpha \left(  t\right)  :[0,\infty)\rightarrow
\lbrack0,\infty)$ is said to be a $\mathcal{K}$ function if $\alpha \left(
0\right)  =0$ and it is non-decreasing.

\item A function $\alpha \left(  t\right)  :[0,\infty)\rightarrow
\lbrack0,\infty)$ is said to be a $\mathcal{K}_{\infty}$ function if
$\alpha \left(  0\right)  =0$, $\lim_{s\rightarrow \infty}\alpha \left(
s\right)  =0$ and it is strictly increasing.

\item A function $\alpha \left(  t\right)  :J\rightarrow(0,\infty)$ is said to
be a $\mathcal{N}$ function if it is positive valued and non-decreasing.

\item A function $\alpha \left(  t,s\right)  :[0,\infty)\times \lbrack
0,\infty)\rightarrow \lbrack0,\infty)$ is said to be a $\mathcal{KL}$ function
if $\alpha \left(  t,\cdot \right)  \in \mathcal{K}$ and $\alpha \left(
\cdot,s\right)  $ is nondecreasing with respect to $s$ and $\lim
_{s\rightarrow \infty}\alpha \left(  t,s\right)  =0.$

\item A function $\alpha \left(  t,s\right)  \in \mathcal{KL}$ is said to be a
$\mathcal{KL_{\infty}}$ function if, moreover, $\alpha \left(  t,\cdot \right)
\in \mathcal{K}_{\infty}$ $.$

\item A function $\alpha \left(  t,s\right)  :J\times \lbrack0,\infty
)\rightarrow \lbrack0,\infty)$ is said to be a $\mathcal{NK}_{\infty}$ function
if $\alpha \left(  t,\cdot \right)  \in \mathcal{N}$ and $\alpha \left(
\cdot,s\right)  \in \mathcal{K}_{\infty}.$
\end{itemize}

\section{\label{sec2}System Description and Preliminaries}

Consider the following nonlinear time-varying system%
\begin{equation}
\dot{x}\left(  t\right)  =f\left(  t,x\left(  t\right)  ,u\left(  t\right)
\right)  , \label{sys}%
\end{equation}
where $f:J\times \mathbf{R}^{n}\times \mathbf{R}^{m}\rightarrow \mathbf{R}^{n}$
is continuous, locally Lipschitz on $x$ for bounded $u$ and such that
$f\left(  t,0,0\right)  =0$. The input $u:J\rightarrow \mathbf{R}^{m}$ is
assumed to be locally essentially bounded. In this note, we are interested in
the stability analysis of this class of systems. Throughout this note, for any
$\mathbb{C}^{1}$ function $V:J\times \mathbf{R}^{n}\rightarrow \mathbf{R},$ we
denote%
\begin{equation}
\left.  \dot{V}\left(  t,x\right)  \right \vert _{(\ref{sys})}\triangleq
\frac{\partial V\left(  t,x\right)  }{\partial t}+\frac{\partial V\left(
t,x\right)  }{\partial x}f\left(  t,x,u\right)  . \label{eq6}%
\end{equation}

Next we introduce the concept of stable functions proposed in
\cite{zhou14auto}. Consider the following scalar linear time-varying (LTV)
system%
\begin{equation}
\dot{y}(t)=\mu(t)y(t),\;t\in J, \label{scalar}%
\end{equation}
where $y(t):J\rightarrow \mathbf{R}$ is the state variable and$\  \mu
(t)\in \mathbb{PC}\left(  J,\mathbf{R}\right)  $. It is not hard to see that
the state transition matrix for system (\ref{scalar}) is given by%
\begin{equation}
\phi \left(  t,t_{0}\right)  =\exp \left(  \int_{t_{0}}^{t}\mu \left(  s\right)
\mathrm{d}s\right)  ,\; \forall t\geq t_{0}\in J. \label{st}%
\end{equation}

\begin{definition}
\label{df1}\cite{zhou14auto} The function $\mu(t)\in \mathbb{PC}\left(
J,\mathbf{R}\right)  $ is said to be \vspace{-0.3cm}

\begin{enumerate}
\item asymptotically stable if the scalar LTV system (\ref{scalar}) is
asymptotically stable;

\item exponentially stable if the scalar LTV system (\ref{scalar}) is
exponentially stable, namely, there exist constants $k\left(  t_{0}\right)
>0$ and $\alpha>0$ such that
\begin{equation}
\left \vert y\left(  t\right)  \right \vert \leq k\left(  t_{0}\right)
\left \vert y\left(  t_{0}\right)  \right \vert \exp \left(  -\alpha \left(
t-t_{0}\right)  \right)  ,\forall t\geq t_{0}\in J. \label{eq10}%
\end{equation}

\item uniformly exponentially stable (or uniformly asymptotically stable) if
the scalar LTV system (\ref{scalar}) is uniformly exponentially stable,
namely, the constant $k\left(  t_{0}\right)  $ in (\ref{eq10}) is independent
of $t_{0}$.
\end{enumerate}
\end{definition}

In the above definition we have noticed that, for linear system, uniformly
asymptotic stability and uniformly exponential stability are equivalent (see,
for example, \cite{rugh96book}). By noting the transition matrix (\ref{st}),
we can obtain immediately the following fact.

\begin{lemma}
\label{lm0}\cite{zhou14auto} The scalar function $\mu(t)\in \mathbb{PC}\left(
J,\mathbf{R}\right) $ is

\begin{enumerate}
\item asymptotically stable if and only if
\begin{equation}
\lim_{t\rightarrow \infty}\int_{t_{0}}^{t}\mu(s)\mathrm{d}s=-\infty.
\label{eq17}%
\end{equation}

\item exponentially stable if and only if there exist $\beta(t_{0})\geq0$ and
$\alpha>0$ such that%
\begin{equation}
\text{ }\int_{t_{0}}^{t}\mu(s)\mathrm{d}s\leq-\alpha(t-t_{0})+\beta
(t_{0}),\  \forall t\geq t_{0}\in J. \label{eqadd2}%
\end{equation}

\item uniformly exponentially stable if and only if (\ref{eqadd2}) is
satisfied where $\beta$ is independent of $t_{0}.$
\end{enumerate}
\end{lemma}

Of course, if $\mu(t)\in \mathbb{PC}\left(  J,\mathbf{R}\right)  $ is a
periodic function with period $T$, then it is easy to see that the three
different stability concepts in Definition \ref{df1} are equivalent, and
moreover they are equivalent to the existence of $c>0$ such that
\cite{zhou14auto}
\begin{equation}
\int_{t}^{t+T}\mu(s)\mathrm{d}s\leq-c. \label{eqadd8}%
\end{equation}

\section{\label{sec3}Main Results}

\subsection{Asymptotic Stability Analysis}

We first give the following definition.

\begin{definition}
\label{df2}The nonlinear time-varying system (\ref{sys}) is said to be

\begin{enumerate}
\item globally asymptotically stable if, for any $\varepsilon>0,$ there is
$\delta \left(  t_{0},\varepsilon \right)  >0$ such that $\left \vert x\left(
t_{0}\right)  \right \vert \leq \delta \left(  t_{0},\varepsilon \right)
\Rightarrow \left \vert x\left(  t\right)  \right \vert \leq \varepsilon,\forall
t\geq t_{0}\in J$ (Lyapunov stability), and, for any $x\left(  t_{0}\right)
\in \mathbf{R}^{n},$ there holds $\lim_{t\rightarrow \infty}\left \vert x\left(
t\right)  \right \vert =0$ (attractivity) \cite{khalil02book}, or equivalently,
there exists a $\sigma \in \mathcal{KL}$ and a $\theta \in \mathcal{N}$ such that,
for any $x\left(  t_{0}\right)  \in \mathbf{R}^{n}$ (Proposition 2.5 in
\cite{kt03siam}),
\[
\left \vert x\left(  t\right)  \right \vert \leq \sigma \left(  \theta \left(
t_{0}\right)  \left \vert x\left(  t_{0}\right)  \right \vert ,t-t_{0}\right)
,\  \forall t\geq t_{0}\in J.
\]

\item globally uniformly asymptotically stable if there exists a $\sigma
\in \mathcal{KL}$ such that, for any $x\left(  t_{0}\right)  \in \mathbf{R}^{n}$
\cite{khalil02book},
\[
\left \vert x\left(  t\right)  \right \vert \leq \sigma \left(  \left \vert
x\left(  t_{0}\right)  \right \vert ,t-t_{0}\right)  ,\  \forall t\geq t_{0}\in
J.
\]

\item globally exponentially stable if there is a $\theta \in \mathcal{N}$ and
$\alpha>0$ such that \cite{zhou14auto}
\begin{equation}
\left \vert x\left(  t\right)  \right \vert \leq \theta \left(  t_{0}\right)
\left \vert x\left(  t_{0}\right)  \right \vert \exp \left(  -\alpha \left(
t-t_{0}\right)  \right)  ,\; \forall t\geq t_{0}\in J. \label{eq12}%
\end{equation}

\item globally uniformly exponentially stable if (\ref{eq12}) is satisfied
with $\theta \left(  t_{0}\right)  $ independent of $t_{0}$ \cite{khalil02book}.
\end{enumerate}
\end{definition}

To the best of our knowledge, in the literature the exponential stability
refers to uniformly exponential stability, while the non-uniformly exponential
stability concept in Item 3 is not well recognized, and was only emphasized
recently in our work \cite{zhou15auto} for LTV systems.

\begin{theorem}
\label{th1}Assume that there exists a $\mathbf{C}^{1}$ function $V:J\times
\mathbf{R}^{n}\rightarrow \lbrack0,\infty)$, two $\mathcal{NK}_{\infty}$
functions $\alpha_{i},i=1,2,$ and a scalar function $\mu \left(  t\right)
\in \mathbb{PC}\left(  J,\mathbf{R}\right)  $ such that, for all $t\in J$ and
$x\in \mathbf{R}^{n},$%
\begin{align}
\alpha_{1}\left(  t,\left \vert x\right \vert \right)   &  \leq V\left(
t,x\right)  \leq \alpha_{2}\left(  t,\left \vert x\right \vert \right)
,\label{eq20}\\
\left.  \dot{V}\left(  t,x\right)  \right \vert _{(\ref{sys})\text{ where
}u\equiv0}  &  \leq \mu \left(  t\right)  V\left(  t,x\right)  , \label{eq21}%
\end{align}
are satisfied. Then the nonlinear system (\ref{sys}) with $u\equiv0$ is

\begin{enumerate}
\item globally asymptotically stable if $\mu \left(  t\right)  $ is
asymptotically stable.

\item globally uniformly asymptotically stable if $\mu \left(  t\right)  $ is
uniformly exponentially stable and $\alpha_{i}\left(  t,s\right)  ,i=1,2$ are
independent of $t.$

\item globally exponentially stable if $\mu \left(  t\right)  $ is
exponentially stable and there exist $m>0$ and $k_{i}\left(  \cdot \right)
\in \mathcal{N},i=1,2$ such that $\alpha_{i}\left(  t,s\right)  =k_{i}\left(
t\right)  s^{m},i=1,2.$

\item globally uniformly exponentially stable if $\mu \left(  t\right)  $ is
uniformly exponentially stable and there exist $m>0,k_{i}>,i=1,2$ such that
$\alpha_{i}\left(  t,s\right)  =k_{i}s^{m},i=1,2.$
\end{enumerate}
\end{theorem}

\begin{proof}
Notice that (\ref{eq21}) implies $\frac{\mathrm{d}}{\mathrm{d}t}\ln V\left(
t,x\right)  =\frac{\dot{V}\left(  t,x\right)  }{V\left(  t,x\right)  }\leq
\mu \left(  t\right)  ,\forall t\geq t_{0}\in J,$ from which it follows that%
\begin{align}
\alpha_{1}\left(  t_{0},\left \vert x\left(  t\right)  \right \vert \right)   &
\leq \alpha_{1}\left(  t,\left \vert x\left(  t\right)  \right \vert \right)
\nonumber \\
&  \leq V\left(  t,x\left(  t\right)  \right)  )\nonumber \\
&  \leq V\left(  t_{0},x\left(  t_{0}\right)  \right)  \phi \left(
t,t_{0}\right) \nonumber \\
&  \leq \alpha_{2}\left(  t_{0},\left \vert x\left(  t_{0}\right)  \right \vert
\right)  \phi \left(  t,t_{0}\right)  . \label{zbzb1}%
\end{align}

\textit{Proof of Item 1}: \ Since $\lim_{t\rightarrow \infty}\phi \left(
t,t_{0}\right)  =0,$ there exists a $T=T\left(  t_{0}\right)  $ such that
$\phi \left(  t,t_{0}\right)  \leq1,t\geq t_{0}+T\left(  t_{0}\right)  .$ Let%
\[
\gamma \left(  t_{0}\right)  \triangleq \max_{s\in \left[  t_{0},t_{0}+T\left(
t_{0}\right)  \right]  }\left \{  \phi \left(  t,t_{0}\right)  \right \}  \geq1.
\]
Then we have from (\ref{zbzb1}) that%
\begin{equation}
\alpha_{1}(t_{0},\left \vert x\left(  t\right)  \right \vert )\leq \alpha
_{2}\left(  t_{0},\left \vert x\left(  t_{0}\right)  \right \vert \right)
\gamma \left(  t_{0}\right)  ,\; \forall t\geq t_{0}. \label{zbzb2}%
\end{equation}
Hereafter, for a function $\alpha \in \mathcal{NK}_{\infty},$ we use
$\alpha^{-1}(t,s)$ denotes the inverse function of $\alpha(t,s)$ with respect
to the second variable, namely, $\alpha^{-1}\left(  t,\alpha \left(
t,s\right)  \right)  \equiv1$. Now we set $\delta \left(  t_{0}\right)
=\alpha_{2}^{-1}(t_{0},\frac{1}{\gamma \left(  t_{0}\right)  }\alpha_{1}%
(t_{0},\varepsilon))$, or equivalently, $\alpha_{2}\left(  t_{0},\delta \left(
t_{0}\right)  \right)  \gamma \left(  t_{0}\right)  =\alpha_{1}(t_{0}%
,\varepsilon).$ Here Then it follows from (\ref{zbzb2}) that, for any
$\left \vert x\left(  t_{0}\right)  \right \vert \leq \delta \left(  t_{0}\right)
,$
\begin{align*}
\alpha_{1}(t_{0},\left \vert x\left(  t\right)  \right \vert )  &  \leq
\alpha_{2}\left(  t_{0},\left \vert x\left(  t_{0}\right)  \right \vert \right)
\gamma \left(  t_{0}\right) \\
&  \leq \alpha_{2}\left(  t_{0},\delta \left(  t_{0}\right)  \right)
\gamma \left(  t_{0}\right) \\
&  =\alpha_{1}(t_{0},\varepsilon),\; \forall t\geq t_{0}\in J,
\end{align*}
which is just $\left \vert x\left(  t\right)  \right \vert \leq \varepsilon
,\forall t\geq t_{0}.$ On the other hand, it follows from (\ref{eq17}) and
(\ref{zbzb1}) that $\lim_{t\rightarrow \infty}\left \vert x\left(  t\right)
\right \vert =0.$ This proves that the system is globally asymptotically stable.

\textit{Proof of Item 2}: It follows from Item 3 of Lemma \ref{lm0} and
(\ref{zbzb1}) that, for any $t\geq t_{0}\in J,$
\begin{align*}
\left \vert x\left(  t\right)  \right \vert  &  \leq \alpha_{1}^{-1}\left(
\alpha_{2}\left(  \left \vert x\left(  t_{0}\right)  \right \vert \right)
\phi \left(  t,t_{0}\right)  \right) \\
&  \leq \alpha_{1}^{-1}\left(  \alpha_{2}\left(  \left \vert x\left(
t_{0}\right)  \right \vert \right)  \exp \left(  \beta \right)  \exp \left(
-\alpha \left(  t-t_{0}\right)  \right)  \right)  \in \mathcal{KL},
\end{align*}
which shows that the system is globally uniformly asymptotically stable.

\textit{Proof of Item 3}: By noting that $\alpha_{1}^{-1}\left(  t,s\right)
=s^{\frac{1}{m}}k_{i}^{-\frac{1}{m}}\left(  t\right)  $ and (\ref{eqadd2}), we
obtain from (\ref{zbzb1}) the following%
\begin{align}
\left \vert x\left(  t\right)  \right \vert  &  \leq k_{1}^{-\frac{1}{m}}\left(
t_{0}\right)  \left(  k_{2}\left(  t_{0}\right)  \left \vert x\left(
t_{0}\right)  \right \vert ^{m}\exp \left(  -\alpha \left(  t-t_{0}\right)
+\beta \left(  t_{0}\right)  \right)  \right)  ^{\frac{1}{m}}\nonumber \\
&  =\left(  \frac{k_{2}\left(  t_{0}\right)  }{k_{1}\left(  t_{0}\right)
}\right)  ^{\frac{1}{m}}\exp \left(  \frac{\beta \left(  t_{0}\right)  }%
{m}\right)  \left \vert x\left(  t_{0}\right)  \right \vert \exp \left(
-\frac{\alpha}{m}\left(  t-t_{0}\right)  \right)  , \label{eq23}%
\end{align}
which indicates that the system is globally exponentially stable.

\textit{Proof of Item 4}: This follows from (\ref{eq23}) since $k_{i},i=1,2$
and $\beta \left(  t_{0}\right)  $ are independent of $t_{0}.$ The proof is finished.
\end{proof}

To go further, we introduce the following technical lemma whose proof is
provided in Appendix A1.

\begin{lemma}
\label{lm2}(Generalized Gronwall-Bellman Inequality) Assume that $\mu \left(
t\right)  ,\pi \left(  t\right)  \in \mathbb{PC}\left(  J,\mathbf{R}\right)  $
and $y\left(  t\right)  :J\rightarrow \lbrack0,\infty)$ be such that%
\begin{equation}
\dot{y}\left(  t\right)  \leq \mu \left(  t\right)  y\left(  t\right)
+\pi \left(  t\right)  ,\;t\in J. \label{eq90}%
\end{equation}
Then, for any $t\geq s\in J,$ the following inequality holds true%
\begin{equation}
y\left(  t\right)  \leq y\left(  s\right)  \phi \left(  t,s\right)  +\int
_{s}^{t}\phi \left(  t,\lambda \right)  \pi \left(  \lambda \right)
\mathrm{d}\lambda. \label{eqtemp3}%
\end{equation}

\end{lemma}

The above lemma can be regarded as the Gronwall-Bellman inequality in the
differential form. Notice that differently from the generalized integral
Gronwall-Bellman inequality in \cite{ZCD13IJC}, the function $\mu \left(
t\right)  $ is not required to be positive for all $t.$ \ We then can state
the following theorem which allows a non-negative drifting term in the
time-derivatives of Lyapunov functions.

\begin{theorem}
\label{th3}Assume that there exists a $\mathbf{C}^{1}$ function $V:J\times
\mathbf{R}^{n}\rightarrow \lbrack0,\infty)$, two $\mathcal{NK}_{\infty}$
functions $\alpha_{i},i=1,2,$ an asymptotically stable function $\mu \left(
t\right)  \in \mathbb{PC}\left(  J,\mathbf{R}\right)  $, and a scalar function
$\pi \left(  t\right)  \in \mathbb{PC}\left(  J,[0,\infty)\right)  $ such that,
for all $\left(  t,x\right)  \in J\times \mathbf{R}^{n},$ (\ref{eq20}) and the
following inequality%
\begin{equation}
\left.  \dot{V}\left(  t,x\right)  \right \vert _{(\ref{sys})\text{ where
}u\equiv0}\leq \mu \left(  t\right)  V\left(  t,x\right)  +\pi \left(  t\right)
, \label{eq91}%
\end{equation}
are satisfied. Let $\phi \left(  t,s\right)  $ be defined in (\ref{st}) and
denote $\kappa \left(  t,t_{0}\right)  :J\times J\rightarrow \mathbf{R}$ as
\begin{equation}
\kappa \left(  t,t_{0}\right)  =\int_{t_{0}}^{t}\phi \left(  t,s\right)
\pi \left(  s\right)  \mathrm{d}s. \label{eqk}%
\end{equation}
Then the nonlinear time-varying system (\ref{sys}) is globally asymptotically
stable if $\kappa \left(  t,t_{0}\right)  $ is bounded for any $t\geq t_{0}\in
J$ and
\begin{equation}
\lim_{t\rightarrow \infty}\kappa \left(  t,t_{0}\right)  =\lim_{t\rightarrow
\infty}\int_{t_{0}}^{t}\phi \left(  t,s\right)  \pi \left(  s\right)
\mathrm{d}s=0,\  \forall t_{0}\in J. \label{eq93}%
\end{equation}

\end{theorem}

\begin{proof}
Applying Lemma \ref{lm2} on inequality (\ref{eq91}) gives, for all $t\geq
t_{0}\in J,$%
\begin{align*}
\alpha_{1}\left(  t_{0},\left \vert x\left(  t\right)  \right \vert \right)   &
\leq \alpha_{1}\left(  t,\left \vert x\left(  t\right)  \right \vert \right) \\
&  \leq V\left(  t,x\left(  t\right)  \right) \\
&  \leq V\left(  t_{0},x\left(  t_{0}\right)  \right)  \phi \left(
t,t_{0}\right)  +\int_{t_{0}}^{t}\phi \left(  t,t_{0}\right)  \pi \left(
s\right)  \mathrm{d}s\\
&  =V\left(  t_{0},x\left(  t_{0}\right)  \right)  \phi \left(  t,t_{0}\right)
+\kappa \left(  t,t_{0}\right)  .
\end{align*}
Hence, by noting that $\lim_{t\rightarrow \infty}\phi \left(  t,t_{0}\right)
=0$ and (\ref{eq93}), we have $\lim_{t\rightarrow \infty}\alpha_{1}\left(
t_{0},\left \vert x\left(  t\right)  \right \vert \right)  =0,$ which in turn
implies $\lim_{t\rightarrow \infty}\left \vert x\left(  t\right)  \right \vert
=0.$ Hence the system is globally attractive. \ On the other hand, as
$\kappa \left(  t,t_{0}\right)  $ and $\phi \left(  t,t_{0}\right)  $ are
bounded, $\alpha_{1}\left(  t_{0},\left \vert x\left(  t\right)  \right \vert
\right)  $ is bounded, which in turn implies that $\left \vert x\left(
t\right)  \right \vert $ is bounded, namely, the system is Lagrange stable.
Then, by Proposition 2.5 in \cite{kt03siam}, the system is Lyapunov stable.
Consequently, the nonlinear time-varying system (\ref{sys}) is globally
asymptotically stable. The proof is finished.
\end{proof}

The most important advantage of Theorem \ref{th3} is that the right hand side
of (\ref{eq91}) is not required to be negative for all time and, moreover, is
even allowed to have a drifting term that is non-negative for all time.

To test the conditions imposed on $\left(  \mu \left(  t\right)  ,\pi \left(
t\right)  \right)  $ in Theorem \ref{th3}, we assume that there exists a
function $\varpi \left(  \cdot \right)  :[0,\infty)\rightarrow(0,\infty)$ such
that%
\begin{equation}
\phi \left(  t,s\right)  =\exp \left(  \int_{s}^{t}\mu \left(  \omega \right)
\mathrm{d}\omega \right)  \leq \varpi \left(  t-s\right)  ,\forall t\geq s\in J.
\label{eq78}%
\end{equation}
This class of function $\omega \left(  s\right)  $ was firstly introduced by
Kalman in \cite{malman60bcmm}, where the function is use to characterize the
uniformly complete controllability concept for LTV systems. Then
$\kappa \left(  t,t_{0}\right)  $ satisfies all the conditions in Theorem
\ref{th3}\ if the function%
\[
\varkappa \left(  t,t_{0}\right)  =\int_{t_{0}}^{t}\varpi \left(  t-s\right)
\pi \left(  s\right)  \mathrm{d}s,\;t\geq t_{0}\in J,
\]
is bounded and such that $\lim_{t\rightarrow \infty}\varkappa \left(
t,t_{0}\right)  =0.$ This can be tested by the following result which can be
regarded as a generalized Gelig lemma \cite{gly78book}. The proof of this
lemma is provided in Appendix A2.

\begin{lemma}
\label{lm3}Consider two functions $\varphi_{1}:[0,\infty)\rightarrow
\mathbf{R},$ $\varphi_{2}:J\rightarrow \mathbf{R}$ and denote $\varphi \left(
t,\tau \right)  =\int_{\tau}^{t}\varphi_{1}\left(  t-s\right)  \varphi
_{2}\left(  s\right)  \mathrm{d}s.$ Then, for any $t\geq \tau \in J,$
$\varphi \left(  t,\tau \right)  $ is uniformly bounded and the following
relation holds true%
\begin{equation}
\lim_{t\rightarrow \infty}\varphi \left(  t,\tau \right)  =\lim_{t\rightarrow
\infty}\int_{\tau}^{t}\varphi_{1}\left(  t-s\right)  \varphi_{2}\left(
s\right)  \mathrm{d}s=0, \label{eqnew0}%
\end{equation}
if one of the following three conditions holds true:

\begin{enumerate}
\item $\varphi_{1}\in \mathcal{L}_{p}\left(  [0,\infty)\right)  ,\varphi_{2}%
\in \mathcal{L}_{q}\left(  J\right)  ,$ where $p,q\in \left(  0,\infty \right)  $
are such that $\frac{1}{p}+\frac{1}{q}=1.$

\item $\varphi_{1}\in \mathcal{L}_{1}\left(  [0,\infty)\right)  ,$ and
$\varphi_{2}\left(  t\right)  \in \mathcal{Z}\left(  J\right)  $.

\item $\varphi_{1}\left(  t\right)  \in \mathcal{Z}\left(  [0,\infty)\right)
,$ and $\varphi_{2}\in \mathcal{L}_{1}\left(  J\right)  $.
\end{enumerate}
\end{lemma}

Notice that the inequality in (\ref{eq78}) is satisfied if $\mu \left(
t\right)  $ is uniformly exponentially stable. In this case $\varpi \left(
s\right)  =\exp \left(  -\alpha s\right)  \in \mathcal{L}_{p}\left(
[0,\infty)\right)  ,p\in \lbrack1,\infty)$ and, by Lemma \ref{lm3}, $\left(
\mu \left(  t\right)  ,\pi \left(  t\right)  \right)  $ satisfies the conditions
in Theorem \ref{th3} if the function $\pi \left(  t\right)  $ satisfies either
$\pi \left(  t\right)  \in \mathcal{Z}\left(  J\right)  $ or $\pi \left(
t\right)  \in \mathcal{L}_{q}\left(  J\right)  ,q\in \lbrack1,\infty).$ This
generalizes Theorem 3.1 in \cite{kt03siam} where $\mu \left(  t\right)  =-1$
and $\pi \left(  t\right)  \in \mathcal{Z}\left(  J\right)  \cap \mathcal{L}%
_{1}\left(  J\right)  .$

\begin{remark}
\label{rm1}In this remark we point out that there is no need to assume that
$\mu \left(  t\right)  $ is uniformly exponentially stable in Theorem
\ref{th2}. Let $J=[0,\infty),\mu \left(  t\right)  =\frac{-2t}{1+t^{2}}$ and
$\pi \left(  t\right)  =\frac{1}{1+t^{2}}\in \mathcal{Z}\left(  J\right)
\cap \mathcal{L}_{1}\left(  J\right)  .$ It follows from $\phi \left(
t,s\right)  =\frac{1+s^{2}}{1+t^{2}}\ $that $\mu \left(  t\right)  $ is not
(uniformly) exponentially stable. However, it follows from $\kappa \left(
t,t_{0}\right)  =\frac{t-t_{0}}{1+t^{2}}$ that the conditions in Theorem
\ref{th3} are satisfied with this pair of $\left(  \mu \left(  t\right)
,\pi \left(  t\right)  \right)  $.
\end{remark}

\subsection{Input-to-State Stability Analysis}

The concept of ISS was introduced by E. D. Sontag in the later 1980's
\cite{sontag89tac} and has been served as a fundamental tool in the analysis
and design of nonlinear systems, such as observer design, small gain theorem,
and stability test of connected nonlinear systems
\cite{kt03siam,mm05auto,nhwls12scl,sontag98scl}. In the literature, the ISS
property is frequently characterized by the ISS-Lypunov function
\cite{sontag98scl}. As usual, the time-derivative of the ISS-Lypunov function
is required to be negative definite under some additional condition on the
input signal $u.$ In this subsection, we will show how to utilize the idea in
the above subsection to deal with ISS stability analysis of nonlinear
time-varying systems by allowing indefinite time-derivatives for the
ISS-Lyapunov functions.

\begin{definition}
The nonlinear system (\ref{sys}) is said to be

\begin{enumerate}
\item input-to-state stable (IIS) if there exist $\sigma \in \mathcal{KL}$ and
$\gamma_{1}\in \mathcal{K}$ such that, for any $u\in$ $\mathcal{L}_{\infty}%
^{m},$ (see Eq. (5') in \cite{sontag98scl})%
\[
\left \vert x\left(  t\right)  \right \vert \leq \sigma \left(  \left \vert
x\left(  t_{0}\right)  \right \vert ,t-t_{0}\right)  +\gamma_{1}\left(
\left \Vert u\right \Vert _{\left[  t_{0},t\right]  }\right)  ,\; \forall t\geq
t_{0}\in J.
\]

\item integral input-to-state stable (iIIS) if there exist $\sigma
\in \mathcal{KL}$ and $\gamma_{1},\gamma_{2}\in \mathcal{K}$ such that (see Eq.
(7) in \cite{asw00tac})%
\[
\left \vert x\left(  t\right)  \right \vert \leq \sigma \left(  \left \vert
x\left(  t_{0}\right)  \right \vert ,t-t_{0}\right)  +\gamma_{1}\left(
\int_{t_{0}}^{t}\gamma_{2}\left(  \left \vert u\left(  s\right)  \right \vert
\right)  \mathrm{d}s\right)  ,\; \forall t\geq t_{0}\in J.
\]
\end{enumerate}
\end{definition}

We first present the following result regarding the characterization of IIS.

\begin{theorem}
\label{th4}Assume that there exist a $\mathbf{C}^{1}$ function $V:J\times
\mathbf{R}^{n}\rightarrow \lbrack0,\infty)$, two $\mathcal{K}_{\infty}$
functions $\alpha_{i},i=1,2,$ a $\mathcal{K}$ function $\rho$, and a uniformly
exponentially stable function $\mu \left(  t\right)  \in \mathbb{PC}\left(
J,\mathbf{R}\right)  $ such that, for all $\left(  t,x\right)  \in
J\times \mathbf{R}^{n},$%
\begin{align}
\alpha_{1}\left(  \left \vert x\right \vert \right)   &  \leq V\left(
t,x\right)  \leq \alpha_{2}\left(  \left \vert x\right \vert \right)
,\label{eq1}\\
\left.  \dot{V}\left(  t,x\right)  \right \vert _{(\ref{sys})}  &  \leq
\mu \left(  t\right)  V\left(  t,x\right)  \text{ if }V\left(  t,x\left(
t\right)  \right)  \geq \rho \left(  \left \vert u\left(  t\right)  \right \vert
\right)  . \label{zbzb4}%
\end{align}
Then the nonlinear system (\ref{sys}) is IIS.
\end{theorem}

\begin{proof}
Let us consider the inequality
\begin{equation}
V\left(  s,x\left(  s\right)  \right)  \geq \rho \left(  \left \vert u\left(
s\right)  \right \vert \right)  . \label{zbzb3}%
\end{equation}
If (\ref{zbzb3}) is true for almost all $s\in \left[  t_{0},t\right]  \subset
J,$ then it follows from (\ref{zbzb4}) that%
\begin{equation}
V\left(  t,x\left(  t\right)  \right)  \leq V\left(  t_{0},x\left(
t_{0}\right)  \right)  \phi \left(  t,t_{0}\right)  \leq \alpha_{2}\left(
\left \vert x\left(  t_{0}\right)  \right \vert \right)  \mathrm{e}^{\beta
}\mathrm{e}^{-\alpha \left(  t-t_{0}\right)  }. \label{zbzb6}%
\end{equation}
Now we assume that (\ref{zbzb3}) does not hold true for almost all
$s\in \left[  t_{0},t\right]  \subset J.$ Let $\{s\in \left[  t_{0},t\right]
:V\left(  s,x\left(  s\right)  \right)  \leq \rho \left(  \left \vert u\left(
s\right)  \right \vert \right)  \}$ which is non-empty. Denote $t^{\ast}%
=\sup \{s\in \left[  t_{0},t\right]  :V\left(  s,x\left(  s\right)  \right)
\leq \rho \left(  \left \vert u\left(  s\right)  \right \vert \right)  \}.$ Then
we have either $t^{\ast}=t$ or $t^{\ast}<t.$ If $t^{\ast}=t,$ it follows from
the definition of $t^{\ast}$ that%
\begin{align}
V\left(  t,x\left(  t\right)  \right)   &  =V\left(  t^{\ast},x\left(
t^{\ast}\right)  \right)  \leq \rho \left(  \left \vert u\left(  t^{\ast}\right)
\right \vert \right) \nonumber \\
&  \leq \sup_{s\in \left[  t_{0},t\right]  }\left \{  \rho \left(  \left \vert
u\left(  s\right)  \right \vert \right)  \right \}  =\rho \left(  \left \Vert
u\left(  t\right)  \right \Vert _{\left[  t_{0},t\right]  }\right)  .
\label{zbzb7}%
\end{align}
If $t^{\ast}<t,$ then $V\left(  s,x\left(  s\right)  \right)  \geq \rho \left(
\left \vert u\left(  s\right)  \right \vert \right)  ,s\in \lbrack t^{\ast},t],$
which, by (\ref{zbzb4}), implies
\[
\left.  \dot{V}\left(  s,x\left(  s\right)  \right)  \right \vert
_{(\ref{sys})}\leq \mu \left(  s\right)  V\left(  s,x\left(  s\right)  \right)
,s\in \left[  t^{\ast},t\right]  ,
\]
from which it follows that%
\begin{align}
V\left(  t,x\left(  t\right)  \right)   &  \leq V\left(  t^{\ast},x\left(
t^{\ast}\right)  \right)  \phi \left(  t,t^{\ast}\right) \nonumber \\
&  =\rho \left(  \left \vert u\left(  t^{\ast}\right)  \right \vert \right)
\phi \left(  t,t^{\ast}\right) \nonumber \\
&  \leq \rho \left(  \left \Vert u\left(  t\right)  \right \Vert _{\left[
t_{0},t\right]  }\right)  \mathrm{e}^{\beta}. \label{zbzb8}%
\end{align}
Hence we get from (\ref{zbzb6}), (\ref{zbzb7}) and (\ref{zbzb8}) that%
\[
V\left(  t,x\left(  t\right)  \right)  \leq \alpha_{2}\left(  \left \vert
x\left(  t_{0}\right)  \right \vert \right)  \mathrm{e}^{\beta}\mathrm{e}%
^{-\alpha \left(  t-t_{0}\right)  }+\rho \left(  \left \Vert u\left(  t\right)
\right \Vert _{\left[  t_{0},t\right]  }\right)  \mathrm{e}^{\beta}.
\]
Hence, by using $\alpha \left(  a+b\right)  \leq \alpha \left(  2a\right)
+\alpha \left(  2b\right)  ,\alpha \in \mathcal{K},a\geq0,b\geq0$, we get%
\begin{align*}
\left \vert x\left(  t\right)  \right \vert  &  \leq \alpha_{1}^{-1}\left(
V\left(  t,x\left(  t\right)  \right)  \right) \\
&  \leq \alpha_{1}^{-1}\left(  2\mathrm{e}^{\beta}\alpha_{2}\left(  \left \vert
x\left(  t_{0}\right)  \right \vert \right)  \mathrm{e}^{-\alpha \left(
t-t_{0}\right)  }\right)  +\alpha_{1}^{-1}\left(  2\mathrm{e}^{\beta}%
\rho \left(  \left \Vert u\left(  t\right)  \right \Vert _{\left[  t_{0}%
,t\right]  }\right)  \right)  ,
\end{align*}
which shows that the system is IIS. The proof is finished.
\end{proof}

\begin{remark}
\label{rm4}Theorem \ref{th4} generalizes the results in \cite{sontag89tac} and
\cite{nhwls12scl}. Particularly, Theorem \ref{th4} improves Theorem 1 in
\cite{nhwls12scl}, where the corresponding function $\mu \left(  t\right)  $
needs to satisfy
\begin{equation}
\int_{t_{0}}^{\infty}\max \{ \mu \left(  s\right)  ,0\} \mathrm{d}s<\infty.
\label{eq123}%
\end{equation}
The above condition is quite restrictive since any piece-wise continuous
periodic function $\mu \left(  t\right)  $ satisfying (\ref{eq123}) if and only
if $\mu \left(  t\right)  \leq0,t\in J,$ namely, condition (\ref{zbzb4})
implies $\dot{V}\left(  t,x\right)  |_{(\ref{sys})}\leq0.$ In this case, by
Theorem 1 in \cite{mazenc03auto}, a new Lyapunov function with negative
definite time-derivative can be constructed instead.
\end{remark}

We next present the following result regarding the characterization of iIIS.

\begin{theorem}
\label{th2}Assume that there exist a $\mathbf{C}^{1}$ function $V:J\times
\mathbf{R}^{n}\rightarrow \lbrack0,\infty)$, two $\mathcal{K}_{\infty}$
functions $\alpha_{i},i=1,2,$ two $\mathcal{K}$ functions $\rho_{i},i=1,2$,
and a uniformly exponentially stable function $\mu \left(  t\right)
\in \mathbb{PC}\left(  J,\mathbf{R}\right)  $ such that, for all $\left(
t,x\right)  \in J\times \mathbf{R}^{n},$ (\ref{eq1}) and the following
inequality are satisfied%
\begin{equation}
\left.  \dot{V}\left(  t,x\right)  \right \vert _{(\ref{sys})}\leq \left(
\rho_{1}\left(  \left \vert u\right \vert \right)  +\mu \left(  t\right)
\right)  V\left(  t,x\right)  +\rho_{2}\left(  \left \vert u\right \vert
\right)  . \label{eq3}%
\end{equation}
Then the nonlinear time-varying system (\ref{sys}) is iIIS with $\gamma
_{1}=\alpha_{1}^{-1}\circ2\pi_{2},\gamma_{2}=\rho=\rho_{1}\vee \rho_{2},$ and
$\sigma \left(  s,t\right)  =\alpha_{1}^{-1}\left(  2\pi_{1}\left(  \alpha
_{2}\left(  s\right)  \mathrm{e}^{\beta}\mathrm{e}^{-\alpha t}\right)
\right)  $ where $\left(  \alpha,\beta \right)  $ is defined in Lemma \ref{lm0}
and%
\[
\pi_{1}\left(  s\right)  =s+\frac{1}{2}s^{2},\; \pi_{2}\left(  s\right)
=\frac{1}{2}\left(  \mathrm{e}^{s}-1\right)  ^{2}+s\mathrm{e}^{\beta s}.
\]

\end{theorem}

\begin{proof}
By the differential form of the Gronwall-Bellman inequality in Lemma \ref{lm2}
one gets from (\ref{eq3}) that%
\begin{align}
V\left(  t,x\left(  t\right)  \right)  \leq &  V\left(  t_{0},x\left(
t_{0}\right)  \right)  \exp \left(  \int_{t_{0}}^{t}\left(  \rho_{1}\left(
\left \vert u\left(  s\right)  \right \vert \right)  +\mu \left(  s\right)
\right)  \mathrm{d}s\right) \nonumber \\
&  +\int_{t_{0}}^{t}\exp \left(  \int_{s}^{t}\left(  \rho_{1}\left(  \left \vert
u\left(  \lambda \right)  \right \vert \right)  +\mu \left(  \lambda \right)
\right)  \mathrm{d}\lambda \right)  \rho_{2}\left(  \left \vert u\left(
s\right)  \right \vert \right)  \mathrm{d}s\nonumber \\
\leq &  V\left(  t_{0},x\left(  t_{0}\right)  \right)  \phi \left(
t,t_{0}\right)  \exp \left(  \int_{t_{0}}^{t}\rho_{1}\left(  \left \vert
u\left(  s\right)  \right \vert \right)  \mathrm{d}s\right) \nonumber \\
&  +\int_{t_{0}}^{t}\exp \left(  \beta \int_{s}^{t}\rho_{1}\left(  \left \vert
u\left(  \lambda \right)  \right \vert \right)  \mathrm{d}\lambda \right)
\rho_{2}\left(  \left \vert u\left(  s\right)  \right \vert \right)
\mathrm{d}s\nonumber \\
\leq &  V\left(  t_{0},x\left(  t_{0}\right)  \right)  \phi \left(
t,t_{0}\right)  \exp \left(  \int_{t_{0}}^{t}\rho \left(  \left \vert u\left(
s\right)  \right \vert \right)  \mathrm{d}s\right) \nonumber \\
&  +\exp \left(  \beta \int_{t_{0}}^{t}\rho \left(  \left \vert u\left(
\lambda \right)  \right \vert \right)  \mathrm{d}\lambda \right)  \int_{t_{0}%
}^{t}\rho \left(  \left \vert u\left(  s\right)  \right \vert \right)
\mathrm{d}s,\; \forall t\geq t_{0}\in J, \label{eq4}%
\end{align}
where we have noticed that $\int_{t_{0}}^{t}\mu \left(  s\right)
\mathrm{d}s\leq \beta,t\geq t_{0},$ and $\int_{s}^{t}\rho_{1}\left(  \left \vert
u\left(  \lambda \right)  \right \vert \right)  \mathrm{d}\lambda \leq \int
_{t_{0}}^{t}\rho_{1}\left(  \left \vert u\left(  \lambda \right)  \right \vert
\right)  \mathrm{d}\lambda,s\in \lbrack t_{0},t].$ Letting $a=V\left(
t_{0},x\left(  t_{0}\right)  \right)  \phi \left(  t,t_{0}\right)
,b=\int_{t_{0}}^{t}\rho \left(  \left \vert u\left(  s\right)  \right \vert
\right)  \mathrm{d}s$ and using the inequality \cite{sontag98scl}
\[
a\mathrm{e}^{b}=a+a\left(  \mathrm{e}^{b}-1\right)  \leq a+\frac{1}{2}%
a^{2}+\frac{1}{2}\left(  \mathrm{e}^{b}-1\right)  ^{2},
\]
to give%
\begin{equation}
V\left(  t,x\left(  t\right)  \right)  \leq \pi_{1}\left(  a\right)  +\pi
_{2}\left(  b\right)  ,\; \forall t\geq t_{0}\in J. \label{eq4a}%
\end{equation}
By using (\ref{eq1}) we further get
\begin{align*}
\pi_{1}\left(  a\right)   &  \leq \pi_{1}\left(  \alpha_{2}\left(  \left \vert
x\left(  t_{0}\right)  \right \vert \right)  \exp \left(  -\alpha \left(
t-t_{0}\right)  +\beta \right)  \right) \\
&  =\pi_{1}\left(  \alpha_{2}\left(  \left \vert x\left(  t_{0}\right)
\right \vert \right)  \mathrm{e}^{\beta}\mathrm{e}^{-\alpha \left(
t-t_{0}\right)  }\right)  ,\; \forall t\geq t_{0}\in J.
\end{align*}
Hence we can obtain from (\ref{eq4a}) the following%
\begin{align*}
\left \vert x\left(  t\right)  \right \vert \leq &  \alpha_{1}^{-1}\left(
V\left(  t,x\left(  t\right)  \right)  \right) \\
\leq &  \alpha_{1}^{-1}\left(  \pi_{1}\left(  a\right)  +\pi_{2}\left(
b\right)  \right) \\
\leq &  \alpha_{1}^{-1}\left(  2\pi_{1}\left(  a\right)  \right)  +\alpha
_{1}^{-1}\left(  2\pi_{2}\left(  b\right)  \right) \\
\leq &  \alpha_{1}^{-1}\left(  2\pi_{1}\left(  \alpha_{2}\left(  \left \vert
x\left(  t_{0}\right)  \right \vert \right)  \mathrm{e}^{\beta}\mathrm{e}%
^{-\alpha \left(  t-t_{0}\right)  }\right)  \right)  +\alpha_{1}^{-1}\left(
2\pi_{2}\left(  \int_{t_{0}}^{t}\rho \left(  \left \vert u\left(  s\right)
\right \vert \right)  \mathrm{d}s\right)  \right)  ,\; \forall t\geq t_{0}\in
J,
\end{align*}
which is the desired result. The proof is finished.
\end{proof}

Letting $\rho_{1}=0$ and $\rho_{2}=\rho$ in Theorem \ref{th2} gives the
following corollary.

\begin{corollary}
\label{coro1}Assume that there exist a $\mathbf{C}^{1}$ function
$V:J\times \mathbf{R}^{n}\rightarrow \lbrack0,\infty)$, two $\mathcal{K}%
_{\infty}$ functions $\alpha_{i},i=1,2,$ a function $\rho \in \mathcal{K}$, and
a uniformly exponentially stable function $\mu \left(  t\right)  \in
\mathbb{PC}\left(  J,\mathbf{R}\right)  $ such that, for all $\left(
t,x\right)  \in J\times \mathbf{R}^{n},$ (\ref{eq1}) and the following
inequality%
\[
\left.  \dot{V}\left(  t,x\right)  \right \vert _{(\ref{sys})}\leq \mu \left(
t\right)  V\left(  t,x\right)  +\rho \left(  \left \vert u\right \vert \right)  ,
\]
are satisfied. Then the nonlinear time-varying system (\ref{sys}) is iIIS with
$\gamma_{1}=\alpha_{1}^{-1}\circ2\pi_{2},\gamma_{2}=\rho,$ and $\sigma \left(
s,t\right)  =\alpha_{1}^{-1}\left(  2\pi_{1}\left(  \alpha_{2}\left(
s\right)  \mathrm{e}^{\beta}\mathrm{e}^{-\alpha t}\right)  \right)  $ where
$\left(  \alpha,\beta \right)  $ is defined in Lemma \ref{lm0} and%
\[
\pi_{1}\left(  s\right)  =s,\; \pi_{2}\left(  s\right)  =s\mathrm{e}^{\beta s}.
\]

\end{corollary}

\begin{remark}
Similar to Remark \ref{rm4}, we mention that Theorem \ref{th2} improves
Theorem 2 in \cite{sontag98scl} and Theorem 3 in \cite{nhwls12scl}, and
Corollary \ref{coro1} improves Theorem 3 in \cite{sontag98scl} and Theorem 2
in \cite{nhwls12scl}, in the sense that the function $\mu \left(  t\right)  $
can take both negative and positive values, and does not need to satisfy the
restrictive condition (\ref{eq123}).
\end{remark}

\section{\label{sec4}Some Illustrative Examples}

In this section, we provide several numerical examples to demonstrate
effectiveness the proposed stability theorems.

\textbf{Example 1}. Consider the following nonlinear time-varying system%
\begin{equation}
\dot{x}(t)=-\frac{x}{t+\sin x},\;t\in J=(1,\infty). \label{ex1}%
\end{equation}
This system has been considered in Example 5.1 in \cite{liao09book}. Choose
$V\left(  x\right)  =x^{2}.$ Then we can compute%
\[
\dot{V}\left(  x\right)  =-\frac{2V}{t+\sin x}\leq-\frac{2V}{t+1}=\mu \left(
t\right)  V\left(  x\right)  ,
\]
where $\mu \left(  t\right)  =-\frac{2}{t+1}.$ It is easy to see that
$\mu \left(  t\right)  $ is an asymptotically stable function and, by Item 1 of
Theorem \ref{th1}, the nonlinear time-varying system (\ref{ex1}) is globally
asymptotically stable. In fact, it follows from $\int_{t_{0}}^{t}\mu \left(
s\right)  \mathrm{d}s=2\ln \frac{1+t_{0}}{1+t}$ that%
\[
\left \vert x\left(  t\right)  \right \vert =V^{\frac{1}{2}}\left(  x\right)
\leq \left \vert x\left(  t_{0}\right)  \right \vert \exp \left(  \frac{1}{2}%
\int_{t_{0}}^{t}\mu \left(  s\right)  \mathrm{d}s\right)  =\left \vert x\left(
t_{0}\right)  \right \vert \frac{1+t_{0}}{1+t},\; \forall t\geq t_{0}\in J.
\]

\textbf{Example 2}: Consider the following planar nonlinear time-varying
system%
\begin{equation}
\left \{
\begin{array}
[c]{c}%
\dot{x}_{1}=-\frac{1}{1+t}x_{1}+t^{2}x_{2}^{2k-1}-tx_{1}^{2r-1},\\
\dot{x}_{2}=-\frac{1}{1+t}x_{2}-t^{2}x_{1}^{2k-1}-tx_{2}^{2r-1},
\end{array}
\right.  \label{ex2}%
\end{equation}
where $t\in J=[0,\infty),$ and $k\geq1,r\geq1$ are two given integers. This
example is a slight modification of Example 5.2 in \cite{liao09book}. We
choose a time-varying Lyapunov function%
\[
V\left(  t,x\right)  =\left(  x_{1}^{2k}+x_{2}^{2k}\right)  \left(
1+t\right)  .
\]
Then the time-derivative of $V\left(  t,x\right)  $ can be evaluated as%
\begin{align*}
\dot{V}\left(  t,x\right)   &  =x_{1}^{2k}+x_{2}^{2k}+2k\left(  1+t\right)
\left(  x_{1}^{2k-1}\dot{x}_{1}+x_{2}^{2k-1}\dot{x}_{2}\right) \\
&  =x_{1}^{2k}+x_{2}^{2k}+2k\left(  1+t\right)  \left(  x_{1}^{2k-1}\left(
-\frac{1}{1+t}x_{1}+t^{2}x_{2}^{2k-1}-tx_{1}^{2r-1}\right)  \right) \\
&  \quad+2k\left(  1+t\right)  x_{2}^{2k-1}\left(  -\frac{1}{1+t}x_{2}%
-t^{2}x_{1}^{2k-1}-tx_{2}^{2r-1}\right) \\
&  =x_{1}^{2k}+x_{2}^{2k}-2k\left(  x_{1}^{2k}+x_{2}^{2k}\right)  -2k\left(
1+t\right)  t\left[  x_{1}^{2\left(  k+r-1\right)  }+x_{2}^{2\left(
k+r-1\right)  }\right] \\
&  \leq \left(  1-2k\right)  \left(  x_{1}^{2k}+x_{2}^{2k}\right) \\
&  =-\frac{2k-1}{1+t}V\left(  t,x\right)  .
\end{align*}
Hence all the conditions in Theorem \ref{th1} are satisfied for any integer
$k\geq1$ and thus this system is globally asymptotically stable. Moreover, by
using the inequality $\left(  a+b\right)  ^{p}\leq2^{p-1}\left(  a^{p}%
+b^{p}\right)  $ where $a\geq0,b\geq0$ and $p\geq1$ is an integer, we obtain%
\begin{align*}
\left(  x_{1}^{2}+x_{2}^{2}\right)  ^{k}  &  \leq2^{k-1}\left(  x_{1}%
^{2k}+x_{2}^{2k}\right) \\
&  \leq2^{k-1}V\left(  t,x\left(  t\right)  \right) \\
&  \leq2^{k-1}V\left(  t_{0},x\left(  t_{0}\right)  \right)  \left(
\frac{1+t_{0}}{1+t}\right)  ^{2k-1}\\
&  =2^{k-1}\left(  x_{1}^{2k}\left(  t_{0}\right)  +x_{2}^{2k}\left(
t_{0}\right)  \right)  \left(  1+t_{0}\right)  \left(  \frac{1+t_{0}}%
{1+t}\right)  ^{2k-1}\\
&  \leq2^{k-1}\left(  x_{1}^{2}\left(  t_{0}\right)  +x_{2}^{2}\left(
t_{0}\right)  \right)  ^{k}\frac{\left(  1+t_{0}\right)  ^{2k}}{\left(
1+t\right)  ^{2k-1}},
\end{align*}
from which it follows that%
\[
\left \vert x\left(  t\right)  \right \vert \leq2^{\frac{k-1}{2k}}\left \vert
x\left(  t_{0}\right)  \right \vert \frac{1+t_{0}}{\left(  1+t\right)
^{1-\frac{1}{2k}}},\; \forall t\geq t_{0}\in J,
\]
which gives an estimate of the decay rate of the system.

\textbf{Example 3}: Consider the following nonlinear time-varying system%
\begin{equation}
\dot{x}\left(  t\right)  =-\frac{1+t}{1+t^{2}}x\left(  t\right)  +\frac
{1}{1+t^{2}}\sin \left(  h\left(  x\left(  t\right)  \right)  \right)  ,\;t\in
J=[0,\infty), \label{sys4}%
\end{equation}
where $h\left(  x\right)  :\mathbf{R}\rightarrow \mathbf{R}$ is any locally
Lipschitz continuous function and $h\left(  0\right)  =0$. Consider the
Lyapunov function $V=x^{2}.$ Then%
\begin{align*}
\dot{V}\left(  x\right)   &  =-2\frac{1+t}{1+t^{2}}V\left(  x\right)
+\frac{2}{1+t^{2}}x\sin^{2}\left(  h\left(  x\left(  t\right)  \right)
\right) \\
&  \leq-2\frac{1+t}{1+t^{2}}V\left(  x\right)  +\frac{2}{1+t^{2}}%
x^{2}+2\left(  \frac{1}{1+t^{2}}\right)  ^{2}\frac{\sin^{2}\left(  h\left(
x\left(  t\right)  \right)  \right)  }{\frac{1}{1+t^{2}}}\\
&  \leq-2\frac{1+t}{1+t^{2}}V\left(  x\right)  +\frac{2}{1+t^{2}}V\left(
x\right)  +\frac{2}{1+t^{2}}\\
&  \leq-\frac{2t}{1+t^{2}}V\left(  x\right)  +\frac{2}{1+t^{2}},
\end{align*}
which corresponds to (\ref{eq91}) with $\mu \left(  t\right)  =-\frac
{2t}{1+t^{2}}$ and $\pi \left(  t\right)  =\frac{2}{1+t^{2}}.$ By Remark
\ref{rm1}, this pair of $\left(  \mu \left(  t\right)  ,\pi \left(  t\right)
\right)  $ satisfies all the conditions in Theorem \ref{th3} from which we
conclude that this nonlinear time-varying system is globally asymptotically stable.

\textbf{Example 4}: Consider the following scalar nonlinear time-varying
system%
\begin{equation}
\dot{x}(t)=\left(  \frac{1}{1+t+x^{2}}-t\left \vert \cos t\right \vert \right)
x+\frac{2t\cos \left \vert t\right \vert }{1+x^{2}}u,\;t\in J=[0,\infty).
\label{sys3}%
\end{equation}
This system is different from system (43) in \cite{nhwls12scl} where $\frac
{1}{1+t+x^{2}}$ is replaced by $\frac{1}{1+t^{2}+x^{2}}.$ Let $V\left(
x\right)  =\frac{1}{2}x^{2}.$ Then%
\begin{align}
\dot{V}\left(  x\right)   &  =2\left(  \frac{1}{1+t+x^{2}}-t\left \vert \cos
t\right \vert \right)  V\left(  x\right)  +\frac{2t\left \vert \cos t\right \vert
x}{1+x^{2}}u\nonumber \\
&  \leq2\left(  \frac{1}{1+t}-t\left \vert \cos t\right \vert \right)  V\left(
x\right)  +t\left \vert \cos t\right \vert \left \vert u\right \vert .
\label{eq56}%
\end{align}
Since $\int_{t_{0}}^{\infty}\frac{2\mathrm{d}s}{1+s}=\infty,$ the inequality
(44) in \cite{nhwls12scl} is not satisfied for this modified system and the
approach there in is not directly applicable.

In the following we show that how our result can apply to this system.
Consider the following scalar function
\begin{equation}
\mu \left(  t\right)  =\frac{2}{1+t}-t\left \vert \cos t\right \vert ,\;t\in
J=[0,\infty). \label{eq57}%
\end{equation}
We can show that $\mu \left(  t\right)  $ is uniformly exponentially stable,
namely, it follows from Lemma \ref{lm0} that there exist two positive numbers
$\alpha>0\ $and $\beta>0$ such that%
\begin{equation}
\int_{t_{0}}^{t}\mu \left(  s\right)  \mathrm{d}s\leq-\alpha \left(
t-t_{0}\right)  +\beta,\; \forall t\geq t_{0}\in J. \label{eq58}%
\end{equation}
The proof of the above inequality has been moved to Appendix A3 for clarity.
If $u=0,$ we get from (\ref{eq56}) that $\dot{V}\left(  x\right)  \leq
\mu \left(  t\right)  V\left(  x\right)  ,$ which, by Theorem \ref{th1},
implies that the system is globally uniformly exponentially stable. Moreover,
according to (\ref{eq59}) in Appendix A3, we have%
\[
\left \vert x\left(  t\right)  \right \vert \leq \exp \left(  \ln \left(
1+\frac{3}{2}\pi \right)  +1\right)  \exp \left(  -\frac{2}{3\pi}\left(
t-t_{0}\right)  \right)  \left \vert x\left(  t_{0}\right)  \right \vert ,\;
\forall t\geq t_{0}\in J.
\]
Notice that only asymptotic stability was claimed in \cite{nhwls12scl}. If
$u\leq V\left(  x\right)  ,$ we can also obtain from (\ref{eq56}) that
$\dot{V}\left(  x\right)  \leq \mu \left(  t\right)  V\left(  x\right)  ,$
which, by Theorem \ref{th4}, implies that the system is ISS.

\section{\label{sec5}Conclusion}

This paper has studied stability analysis of nonlinear time-varying systems by
using Lyapunov's second method. Differently from the traditional Lyapunov
approach, the proposed stability theorem does not require that the
time-derivative of the Lyapunov function is negative definite. The stability
analysis is achieved with the help of the comparison principle and the concept
of scalar stable functions. A couple of stability concepts were considered.
These concepts include asymptotic stability, uniformly asymptotic stability,
exponential stability, uniformly exponential stability, input-to-state
stability and integral input-to-state stability. The developed theorems
improves the existing results and their effectiveness were illustrated by some
numerical examples.

\section*{Appendix}

\subsection*{A1: Proof of Lemma \ref{lm2}}

We write (\ref{eq90}) as $\dot{y}\left(  t\right)  -\mu \left(  t\right)
y\left(  t\right)  \leq \pi \left(  t\right)  ,t\in J,$ by using which we can
obtain%
\begin{align*}
&  \frac{\mathrm{d}}{\mathrm{d}\lambda}\left(  y\left(  \lambda \right)
\exp \left(  -\int_{s}^{\lambda}\mu \left(  \omega \right)  \mathrm{d}%
\omega \right)  \right) \\
&  =\left(  \dot{y}\left(  \lambda \right)  -\mu \left(  \lambda \right)
y\left(  \lambda \right)  \right)  \exp \left(  -\int_{s}^{\lambda}\mu \left(
\omega \right)  \mathrm{d}\omega \right) \\
&  \leq \pi \left(  \lambda \right)  \exp \left(  -\int_{s}^{\lambda}\mu \left(
\omega \right)  \mathrm{d}\omega \right)  ,\; \lambda \geq s\in J,
\end{align*}
from which it follows that, for all $t\geq s\in J$%
\begin{align*}
y\left(  t\right)  \exp \left(  -\int_{s}^{t}\mu \left(  \omega \right)
\mathrm{d}\omega \right)  -y\left(  s\right)   &  =\int_{s}^{t}\mathrm{d}%
\left(  y\left(  \lambda \right)  \exp \left(  -\int_{s}^{\lambda}\mu \left(
\omega \right)  \mathrm{d}\omega \right)  \right) \\
&  \leq \int_{s}^{t}\pi \left(  \lambda \right)  \exp \left(  -\int_{s}^{\lambda
}\mu \left(  \omega \right)  \mathrm{d}\omega \right)  \mathrm{d}\lambda.
\end{align*}
As $\exp(-\int_{s}^{t}\mu \left(  \omega \right)  $\textrm{$d$}$\omega)>0,$ the
above inequality can be simplified as%
\begin{align*}
y\left(  t\right)   &  \leq \left(  y\left(  s\right)  +\int_{s}^{t}\pi \left(
\lambda \right)  \exp \left(  -\int_{s}^{\lambda}\mu \left(  \omega \right)
\mathrm{d}\omega \right)  \mathrm{d}\lambda \right)  \exp \left(  \int_{s}^{t}%
\mu \left(  \omega \right)  \mathrm{d}\omega \right) \\
&  =y\left(  s\right)  \exp \left(  \int_{s}^{t}\mu \left(  \omega \right)
\mathrm{d}\omega \right)  +\int_{s}^{t}\pi \left(  \lambda \right)  \exp \left(
\int_{\lambda}^{t}\mu \left(  \omega \right)  \mathrm{d}\omega \right)
\mathrm{d}\lambda,
\end{align*}
which is just (\ref{eqtemp3}).

\subsection*{A2: Proof of Lemma \ref{lm3}}

We only prove Item 1 since Items 2-3 can be proven in quite a similar way.
Since $\varphi_{1}\in \mathcal{L}_{p}\left(  [0,\infty)\right)  ,\varphi_{2}%
\in \mathcal{L}_{q}\left(  J\right)  ,$ there exist two positive constants
$d_{i},i=1,2$ such that%
\begin{equation}
\left(  \int_{0}^{\infty}\left \vert \varphi_{1}\left(  s\right)  \right \vert
^{p}\mathrm{d}s\right)  ^{\frac{1}{p}}\leq d_{1},\; \left(  \int_{t^{\#}%
}^{\infty}\left \vert \varphi_{2}\left(  s\right)  \right \vert ^{q}%
\mathrm{d}s\right)  ^{\frac{1}{q}}\leq d_{2}. \label{eqnew1}%
\end{equation}
Then, for any $t\geq \tau \in J,$ by the Holder inequality, we obtain
\begin{align*}
\left \vert \varphi \left(  t,\tau \right)  \right \vert  &  =\left \vert
\int_{\tau}^{t}\varphi_{1}\left(  t-s\right)  \varphi_{2}\left(  s\right)
\mathrm{d}s\right \vert \\
&  \leq \left(  \int_{\tau}^{t}\left \vert \varphi_{1}\left(  t-s\right)
\right \vert ^{p}\mathrm{d}s\right)  ^{\frac{1}{p}}\left(  \int_{\tau}%
^{t}\left \vert \varphi_{2}\left(  s\right)  \right \vert ^{q}\mathrm{d}%
s\right)  ^{\frac{1}{q}}\\
&  \leq \left(  \int_{0}^{t-\tau}\left \vert \varphi_{1}\left(  s\right)
\right \vert ^{p}\mathrm{d}s\right)  ^{\frac{1}{p}}\left(  \int_{t^{\#}}%
^{t}\left \vert \varphi_{2}\left(  s\right)  \right \vert ^{q}\mathrm{d}%
s\right)  ^{\frac{1}{q}}\\
&  \leq \left(  \int_{0}^{\infty}\left \vert \varphi_{1}\left(  s\right)
\right \vert ^{p}\mathrm{d}s\right)  ^{\frac{1}{p}}\left(  \int_{t^{\#}%
}^{\infty}\left \vert \varphi_{2}\left(  s\right)  \right \vert ^{q}%
\mathrm{d}s\right)  ^{\frac{1}{q}}\\
&  =d_{1}d_{2},
\end{align*}
which shows that $\left \vert \int_{\tau}^{t}\varphi_{1}\left(  t-s\right)
\varphi_{2}\left(  s\right)  \mathrm{d}s\right \vert $ is uniformly bounded for
any $t\geq \tau \in J.$

By the Cauchy convergence theorem, for any $\varepsilon>0,$ it follows from
(\ref{eqnew1}) that there exist $T_{i}=T_{i}\left(  \varepsilon \right)  >0$
such that, for any $t_{12}\geq t_{11}\geq T_{1}\in \lbrack0,\infty)$ and
$t_{22}\geq t_{21}\geq T_{2}\in J,$ there holds
\begin{equation}
\left(  \int_{t_{11}}^{t_{12}}\left \vert \varphi_{1}\left(  s\right)
\right \vert ^{p}\mathrm{d}s\right)  ^{\frac{1}{p}}\leq \varepsilon,\; \left(
\int_{t_{21}}^{t_{22}}\left \vert \varphi_{2}\left(  s\right)  \right \vert
^{q}\mathrm{d}s\right)  ^{\frac{1}{q}}\leq \varepsilon. \label{eqnew2}%
\end{equation}
Now, for any $t>\tau \in J,$ consider%
\[
\int_{\tau}^{t}\varphi_{1}\left(  t-s\right)  \varphi_{2}\left(  s\right)
\mathrm{d}s=\int_{\tau}^{T}\varphi_{1}\left(  t-s\right)  \varphi_{2}\left(
s\right)  \mathrm{d}s+\int_{T}^{t}\varphi_{1}\left(  t-s\right)  \varphi
_{2}\left(  s\right)  \mathrm{d}s,
\]
where $T\geq \tau$ is any constant to be specified. Then, for any $t\geq
T_{1}+T$ , by using the Holder inequality, we obtain%
\begin{align*}
\left \vert \int_{\tau}^{T}\varphi_{1}\left(  t-s\right)  \varphi_{2}\left(
s\right)  \mathrm{d}s\right \vert  &  \leq \left(  \int_{\tau}^{T}\left \vert
\varphi_{1}\left(  t-s\right)  \right \vert ^{p}\mathrm{d}s\right)  ^{\frac
{1}{p}}\left(  \int_{\tau}^{T}\left \vert \varphi_{2}\left(  s\right)
\right \vert ^{q}\mathrm{d}s\right)  ^{\frac{1}{q}}\\
&  =\left(  \int_{t-T}^{t-\tau}\left \vert \varphi_{1}\left(  s\right)
\right \vert ^{p}\mathrm{d}s\right)  ^{\frac{1}{p}}\left(  \int_{\tau}%
^{T}\left \vert \varphi_{2}\left(  s\right)  \right \vert ^{q}\mathrm{d}%
s\right)  ^{\frac{1}{q}}\\
&  \leq \left(  \int_{t-T}^{t-\tau}\left \vert \varphi_{1}\left(  s\right)
\right \vert ^{p}\mathrm{d}s\right)  ^{\frac{1}{p}}\left(  \int_{\tau}^{\infty
}\left \vert \varphi_{2}\left(  s\right)  \right \vert ^{q}\mathrm{d}s\right)
^{\frac{1}{q}}\\
&  \leq d_{2}\varepsilon,
\end{align*}
where we have noticed that $t-\tau \geq t-T\geq T_{1}.$ Similarly, if $T\geq
T_{2},$ then, for any $t\geq T\geq T_{2},$ we have
\begin{align*}
\left \vert \int_{T}^{t}\varphi_{1}\left(  t-s\right)  \varphi_{2}\left(
s\right)  \mathrm{d}s\right \vert  &  \leq \left(  \int_{T}^{t}\left \vert
\varphi_{1}\left(  t-s\right)  \right \vert ^{p}\mathrm{d}s\right)  ^{\frac
{1}{p}}\left(  \int_{T}^{t}\left \vert \varphi_{2}\left(  s\right)  \right \vert
^{q}\mathrm{d}s\right)  ^{\frac{1}{q}}\\
&  =\left(  \int_{0}^{t-T}\left \vert \varphi_{1}\left(  s\right)  \right \vert
^{p}\mathrm{d}s\right)  ^{\frac{1}{p}}\left(  \int_{T}^{t}\left \vert
\varphi_{2}\left(  s\right)  \right \vert ^{q}\mathrm{d}s\right)  ^{\frac{1}%
{q}}\\
&  \leq \left(  \int_{0}^{\infty}\left \vert \varphi_{1}\left(  s\right)
\right \vert ^{p}\mathrm{d}s\right)  ^{\frac{1}{p}}\left(  \int_{T}%
^{t}\left \vert \varphi_{2}\left(  s\right)  \right \vert ^{q}\mathrm{d}%
s\right)  ^{\frac{1}{q}}\\
&  \leq d_{1}\varepsilon.
\end{align*}
Combining these two cases by setting $T=\max \{T_{2},\tau \},$ we obtain, for
all $t\geq T_{1}+T_{2}+\tau,$%
\[
\left \vert \int_{\tau}^{t}\varphi_{1}\left(  t-s\right)  \varphi_{2}\left(
s\right)  \mathrm{d}s\right \vert \leq \left(  d_{1}+d_{2}\right)  \varepsilon,
\]
which just implies (\ref{eqnew0}). The proof is finished.

\subsection*{A3: Proof of Inequality (\ref{eq58})}

Notice that, for any $t\in J,$ we have
\begin{equation}
\int_{t}^{t+\frac{3}{2}\pi}\frac{2}{1+s}\mathrm{d}s=2\ln \frac{1+t+\frac{3}%
{2}\pi}{1+t}=2\ln \left(  1+\frac{\frac{3}{2}\pi}{1+t}\right)  \leq2\ln \left(
1+\frac{3}{2}\pi \right)  . \label{eq51}%
\end{equation}
Now consider three cases.

\begin{itemize}
\item Case 1: There exists a nonnegative integer $k$ such that $t\in
\lbrack2k\pi,2k\pi+1/2\pi].$ Then we can compute%
\begin{align}
\int_{t}^{t+\frac{3}{2}\pi}s\left \vert \cos s\right \vert \mathrm{d}s  &
\geq \int_{2k\pi+\frac{1}{2}\pi}^{2k\pi+\frac{3}{2}\pi}s\left \vert \cos
s\right \vert \mathrm{d}s\nonumber \\
&  =\int_{0}^{\pi}\left(  \sigma+\left(  2k\pi+\frac{1}{2}\pi \right)  \right)
\left \vert \cos \left(  \sigma+\left(  2k\pi+\frac{1}{2}\pi \right)  \right)
\right \vert \mathrm{d}\sigma \nonumber \\
&  =\int_{0}^{\pi}\left(  \sigma+\left(  2k\pi+\frac{1}{2}\pi \right)  \right)
\sin \left(  \sigma \right)  d\sigma \nonumber \\
&  =\left.  \left(  \sin(\sigma)-\sigma \cos \sigma \right)  \right \vert
_{0}^{\pi}-\left.  \left(  2k\pi+\frac{1}{2}\pi \right)  \cos \left(
\sigma \right)  \right \vert _{0}^{\pi}\nonumber \\
&  =\pi+2\left(  2k\pi+\frac{1}{2}\pi \right)  \geq2\pi. \label{eq52}%
\end{align}

\item Case 2: There exists a nonnegative integer $k$ such that $t\in
(2k\pi+1/2\pi,2k\pi+3/2\pi].$ Then%
\begin{align}
\int_{t}^{t+\frac{3}{2}\pi}s\left \vert \cos s\right \vert \mathrm{d}s  &
\geq \int_{2k\pi+\frac{3}{2}\pi}^{2k\pi+\frac{1}{2}\pi+\frac{3}{2}\pi
}s\left \vert \cos s\right \vert \mathrm{d}s\nonumber \\
&  =\int_{0}^{\frac{1}{2}\pi}\left(  \sigma+\left(  2k\pi+\frac{3}{2}%
\pi \right)  \right)  \left \vert \cos \left(  \sigma+\left(  2k\pi+\frac{3}%
{2}\pi \right)  \right)  \right \vert \mathrm{d}\sigma \nonumber \\
&  =\int_{0}^{\frac{1}{2}\pi}\left(  \sigma+\left(  2k\pi+\frac{3}{2}%
\pi \right)  \right)  \sin \left(  \sigma \right)  \mathrm{d}\sigma \nonumber \\
&  =\left.  \left(  \sin(\sigma)-\sigma \cos \sigma \right)  \right \vert
_{0}^{\frac{1}{2}\pi}-\left.  \left(  2k\pi+\frac{3}{2}\pi \right)  \cos \left(
\sigma \right)  \right \vert _{0}^{\frac{1}{2}\pi}\nonumber \\
&  =1+2k\pi+\frac{3}{2}\pi \geq1+\frac{3}{2}\pi. \label{eq53}%
\end{align}

\item Case 3: There exists a nonnegative integer $k$ such that $t\in
(2k\pi+3/2\pi,2k\pi+2\pi]$. Then%
\begin{align}
\int_{t}^{t+\frac{3}{2}\pi}s\left \vert \cos s\right \vert \mathrm{d}s  &
\geq \int_{2k\pi+2\pi}^{2k\pi+\frac{3}{2}\pi+\frac{3}{2}\pi}s\left \vert \cos
s\right \vert \mathrm{d}s\nonumber \\
&  =\int_{0}^{\pi}\left(  \sigma+\left(  2k\pi+2\pi \right)  \right)
\left \vert \cos \left(  \sigma+\left(  2k\pi+2\pi \right)  \right)  \right \vert
\mathrm{d}\sigma \nonumber \\
&  =\int_{0}^{\pi}\left(  \sigma+\left(  2k\pi+2\pi \right)  \right)
\left \vert \cos \left(  \sigma \right)  \right \vert \mathrm{d}\sigma \nonumber \\
&  \geq \int_{0}^{\frac{\pi}{2}}\left(  \sigma+\left(  2k\pi+2\pi \right)
\right)  \cos \left(  \sigma \right)  \mathrm{d}\sigma \nonumber \\
&  =\left.  \left(  \cos \sigma+\sigma \sin \sigma \right)  \right \vert
_{0}^{\frac{\pi}{2}}+\left.  \left(  2k\pi+2\pi \right)  \sin \left(
\sigma \right)  \right \vert _{0}^{\frac{\pi}{2}}\nonumber \\
&  =\frac{\pi}{2}-1+2k\pi+2\pi \nonumber \\
&  \geq2\pi. \label{eq54}%
\end{align}

\end{itemize}

It follows from (\ref{eq51})--(\ref{eq54}) that, for any $t\in J,$ we have%
\begin{align}
\int_{t}^{t+\frac{3}{2}\pi}\mu \left(  s\right)  \mathrm{d}s  &  =\int
_{t}^{t+\frac{3}{2}\pi}\frac{2\mathrm{d}s}{1+s}-\int_{t}^{t+\frac{3}{2}\pi
}s\left \vert \cos s\right \vert \mathrm{d}s\nonumber \\
&  \leq2\ln \left(  1+\frac{3}{2}\pi \right)  -\left(  1+\frac{3}{2}\pi \right)
\nonumber \\
&  <-2. \label{eq55}%
\end{align}
For any $t\geq t_{0}\in J,$ there exists a unique nonnegative integer $k$ such
that $t\in \lbrack t_{0}+k\frac{3}{2}\pi,t_{0}+\left(  k+1\right)  \frac{3}%
{2}\pi).$ Let $t_{i}=t_{0}+i\frac{3}{2}\pi.$ Then it follows from (\ref{eq55})
and (\ref{eq51}) that%
\begin{align}
\int_{t_{0}}^{t}\mu \left(  s\right)  \mathrm{d}s  &  =\sum \limits_{i=0}%
^{k-1}\int_{t_{0}+i\frac{3}{2}\pi}^{t_{0}+\left(  i+1\right)  \frac{3}{2}\pi
}\mu \left(  s\right)  \mathrm{d}s+\int_{t_{0}+k\frac{3}{2}\pi}^{t}\mu \left(
s\right)  \mathrm{d}s\nonumber \\
&  \leq-2k+\int_{t_{0}+k\frac{3}{2}\pi}^{t}\frac{2}{1+s}\mathrm{d}s\nonumber \\
&  \leq-2k+\int_{t_{0}+k\frac{3}{2}\pi}^{t_{0}+\left(  k+1\right)  \frac{3}%
{2}\pi}\frac{2}{1+s}\mathrm{d}s\nonumber \\
&  \leq-2k+2\ln \left(  1+\frac{3}{2}\pi \right) \nonumber \\
&  =-2\sum \limits_{i=0}^{k-1}\frac{t_{i+1}-t_{i}}{\frac{3}{2}\pi}+2\ln \left(
1+\frac{3}{2}\pi \right) \nonumber \\
&  =-\frac{4}{3\pi}\left(  t_{k}-t_{0}\right)  +2\ln \left(  1+\frac{3}{2}%
\pi \right) \nonumber \\
&  =-\frac{4}{3\pi}\left(  t-t_{0}\right)  +2\ln \left(  1+\frac{3}{2}%
\pi \right)  +\frac{4}{3\pi}\left(  t-t_{k}\right) \nonumber \\
&  \leq-\frac{4}{3\pi}\left(  t-t_{0}\right)  +2\ln \left(  1+\frac{3}{2}%
\pi \right)  +2, \label{eq59}%
\end{align}
which is just (\ref{eq58}), namely, $\mu \left(  t\right)  $ is uniformly
exponentially stable.

\end{document}